\documentclass[11pt]{article}
\usepackage[letterpaper,twoside,outer=1.3in,vmargin=1.3in,]{geometry}
\usepackage{amsmath,amssymb,amsthm,amsfonts,enumerate,times,epsfig,color,hyperref, tikz}
\usepackage{cleveref}
\usetikzlibrary{graphs, shapes, calc}
\usepackage{mathrsfs}
\allowdisplaybreaks
\usepackage{algpseudocode}

\newcommand{\ignore}[1]{}
\usepackage{algorithm2e}
\usepackage{doi}

\DeclareMathOperator{\lat}{lat}

\DeclareMathOperator{\lin}{lin}

\DeclareMathOperator{\supp}{support}

\newcommand{\zB}{\mathcal B}

\newcommand{\zL}{\mathcal L}

\newcommand{\Z}{\mathbb{Z}}

\newcommand{\R}{\mathbb{R}}
\newtheorem{theorem}{Theorem}[section]
\newtheorem{CO}[theorem]{Corollary}

\newtheorem{DE}[theorem]{Definition}

\newcounter{claim-nb}[theorem]
\newtheorem{claim}[claim-nb]{Claim}
\newtheorem*{claim*}{Claim}
\newtheorem*{subclaim*}{Subclaim}

\newcounter{claim-nbs}[section]
\newcounter{subclaim-nb}[claim-nbs]
\newenvironment{cproof}
{\begin{proof}
 [Proof of Claim.]
 \vspace{-1.2\parsep}}
{ \end{proof}}

\title{A Polyhedral Perspective on the Perfect Matching Lattice}
\author{Olha Silina\thanks{Department of Mathematical Sciences, Carnegie Mellon University. 
Email: \texttt{osilina@andrew.cmu.edu}}}
\date{\today}

\begin{document}

\maketitle

\begin{abstract}
    We study the perfect matching lattice of a matching covered graph $G$, generated by the incidence vectors of its perfect matchings. Building on results of Lov\'asz and de Carvalho, Lucchesi, and Murty, we give a polynomial-time algorithm based on polyhedral methods that constructs a lattice basis for this lattice consisting of perfect matchings of $G$. By decomposing along certain odd cuts, we reduce the graph into subgraphs whose perfect matching polytopes coincide with their bipartite relaxations (known as \emph{Birkhoff von Neumann graphs}). This yields a constructive polyhedral proof of the existence of such bases and highlights new connections between combinatorial and geometric properties of perfect matchings.
\end{abstract}

\section{Introduction}


Consider a graph $G=(V,E)$ with an even number of vertices. A \emph{perfect matching} is a set of edges $M\subseteq E$ incident with each vertex exactly once. $G$ is \emph{matching covered} if every edge of $G$ belongs to some perfect matching.
The theory of matching covered graphs is a rich area of combinatorics with pioneering results due to Edmonds, Seymour, and Lov\'asz. An object of particular interest is the \emph{perfect matching lattice}, defined as the set of integral linear combinations of the incidence vectors of perfect matchings: \begin{equation*}
    \zL(G): = \lat(\mathcal{M})= \{\sum_{M \in \mathcal{M}}k_M \mathbf{1}_M: k_M \in \Z\},
\end{equation*} where $\mathcal{M}$ denotes the set of all perfect matchings of $G$.
Lov\'asz \cite{lovaszPM} gave a characterization of this lattice based on a structural decomposition of the graph into blocks called \emph{bricks} and \emph{braces}, with the Petersen graph being the only obstruction for this lattice to contain all the integral points in the linear hull of the perfect matchings. 

A natural question is whether the generating set of this lattice, i.e. the perfect matchings of $G$, contains a basis for $\zL(G)$. This question was answered positively by de Carvalho, Lucchesi, and Murty \cite{dimension} who proved it through ear-decompositions of a graph. In this paper, we present an algorithmic analogue of this result, mainly:
\begin{theorem}\label{thm:main}
    There is a polynomial time algorithm that, given a matching covered graph $G$, finds a lattice basis for $\zL(G)$ consisting of incidence vectors of perfect matchings of $G$.
\end{theorem}
The main contribution of this paper is a new approach to basis construction for the perfect matching lattice based on polyhedral theory. In addition to the algorithmic result, we identify a class of graphs where a lattice basis can be obtained by a simple inductive process. Moreover, we connect the combinatorial and polyhedral properties of the perfect matchings, bridging our approach and structural methods of de Carvalho, Lucchesi, and Murty. More broadly, our polyhedral techniques offer new tools that may prove valuable in tackling longstanding open problems in matching theory. 

To begin with, the \emph{perfect matching polytope} $PM(G)$ is the convex hull of its perfect matching incidence vectors and is described by the following set of inequalities (Edmonds, \cite{edmonds1965maximum}):
\begin{equation}\label{eq:pm-general}
    PM(G)=\left\{x\in \R^A_{+}: \begin{array}{cc}
    x(\delta(v))= 1 &  \forall v\in V\\
     x(C)\geq 1 &  \forall C \textmd{ odd cut}
\end{array}\right\}
\end{equation}
Here, a \emph{cut} $C = \delta(U)$ for some $U\subseteq V$ is the set of edges with exactly one endpoint in $U$; the sets $U$ and its complement $V\backslash U$ are the \emph{shores} of $C$. A cut is \emph{odd} if both of its shores have an odd cardinality. 

The main idea of the paper is to utilize odd cut decomposition methods to reduce our graph to a number of graphs whose perfect matching polytope coincides with its \emph{bipartite relaxation}:\begin{equation}\label{eq:pm-bipartite}
    P(G):=\left\{x\in \R^A_{+}: x(\delta(v))= 1  \forall v\in V \right\}.
\end{equation} 
Matching covered graphs for which $PM(G)=P(G)$ are called \emph{Birkhoff von Neumann (BvN)} graphs. All bipartite graphs and some non-bipartite graphs, in particular certain types of bricks (to be defined later) are BvN. Our algorithm for finding a basis for $\zL(G)$ either finds a set of $\dim(P(G))$ linearly independent integral vertices of $P(G)$ or finds a facet of $PM(G)$induced by an odd cut constraint $x(C)\geq 1$. We then use this facet to further decompose the graph into smaller ones until eventually the algorithm successfully finds a basis.

The paper is organized as follows. In Section \ref{sec:prelim} we discuss a way to decompose the graph into simple pieces and show how to reduce the basis construction for $\zL(G)$ to these pieces. In Section \ref{sec:bvn} we describe a basis construction algorithm for \emph{Birkhoff von Neumann (BvN)} graphs based on polyhedral techniques and discuss what happens when applying this algorithm to a non-BvN graph. In Section \ref{sec:put-together} we show how to decompose non-BvN graphs into BvN ones using a specific type of odd cuts in a process similar to the tight cut decomposition. In Section \ref{sec:sep-find} we provide an algorithm to find an appropriate cut to be used in the above mentioned decomposition. Finally, in Appendix \ref{sec:classical} we draw connections between our polyhedral view of the basis and the classical approach based on \emph{ear decompositions}.  

\section{Preliminaries}\label{sec:prelim}
\subsection{Tight cut decomposition}
Let $G=(V,E)$ be a perfect matching covered graph. An odd cut $C$ is \emph{tight} if every perfect matching of $G$ intersects it exactly once. Polyhedrally, these are the odd cuts whose corresponding constraint in (\ref{eq:pm-general}) is an implicit equality. There are two types of odd cuts of particular interest.
\begin{DE}
    A pair $\{u,v\}$ of vertices in a matching covered graph $G$ is a \emph{$2$-separation} if $G-u-v$ has at least two connected components of even cardinality. A set $B\subseteq V$ of at least $2$ vertices is a \emph{barrier} if $G-B$ has exactly $|B|$ odd connected components.
\end{DE}
It is quite straightforward to see that if $\{u,v\}$ is a $2$-separation in a matching covered $G$ with $U_1$ and $U_2$ being the (even) components of $G-u-v$, then $C:=\delta(U_1\cup \{u\})$ and $C':=\delta(U_1\cup \{v\})$ are non-trivial tight cuts. We refer to these as \emph{$2$-separation cuts}.
Similarly, if $B$ is a barrier with $V_1, V_2, \ldots, V_{|B|}$ being odd components of $G-B$, then each of $\delta(V_i)$ is a tight cut. We refer to these types of cuts as \emph{barrier cuts}.
The following result highlights the importance of the above definitions:
\begin{theorem}[Theorem $4.7$ in \cite{edmonds1982brick}]\label{thm:tight-cut-find}
    If a perfect matching covered graph has a non-trivial tight cut then it has a barrier or $2$-separation tight cut.
\end{theorem}
This theorem implies that tight cuts can be found efficiently. Indeed, it suffices to check for $2$-separation and barrier cuts. The former can be identified by going over all pairs of vertices $\{u,v\}$ and counting the number of components in $G-u-v$. For barrier detection, we use the following characterization for \emph{maximal} barriers, i.e. barriers that are not contained in larger barriers: \begin{theorem}[Theorem $3.2$ in \cite{LovaszOntheStr}]\label{thm:barrier-find}
    The maximal barriers are exactly the equivalence classes under the equivalence relation on $V$: $x\sim y$ iff $G-x-y$ has no perfect matching.
\end{theorem}

Once a tight cut is identified, one could \emph{contract} it: for each of the two shores, consider a graph obtained by shrinking it into a single vertex. Thus, we obtain two smaller graphs. Repeating this process for the new set of graphs, we end up with a list of graphs obtained by contracting each shore of a tight cut into a single vertex. This process is known as \emph{tight cut decomposition} and results in a set of matching covered graphs with no tight cuts. Non-bipartite matching covered graphs with no tight cuts are called \emph{bricks}, and bipartite ones are called \emph{braces}. It is also known that the list of bricks and braces obtained as a result of this decomposition does not depend on the choice of tight cuts. Furthermore, any inclusion-wise maximal laminar family of tight cuts results in a brick and brace decomposition \cite{lovaszPM}. Since any laminar family on $|V|$ elements has size at most $n-1$ (we are interested in non-trivial tight cuts), the total number of bricks and braces is at most $n$, giving the following result:
\begin{CO}
    A tight cut decomposition of a graph can be found in polynomial time.
\end{CO}

\subsection{Lattice theory and dimension results}

Given a finite set of vectors $S =\{\mathbf{v}_1, \ldots, \mathbf{v}_k\} \subset \R^d$, their \emph{linear hull} is the set of all possible linear combinations $\lin(S) = \{\sum_{i=1}^k \alpha_i \mathbf{v}_i : \alpha_i \in \R\}$. An \emph{integral basis} for $\lin(S)$ is a set $B$ of linearly independent vectors such that any $\mathbf{z}\in \lin(S)\cap \Z^d$ is an integral linear combination of $B$. Assuming all the vectors have rational coordinates, the \emph{lattice} spanned by $S$ is the set of all integer linear combinations $\lat(S)=\{\sum_{i=1}^k \alpha_i \mathbf{v}_i: \alpha_i \in \Z\}$. In this case, we say that $S$ \emph{spans} $\lat(S)$. It is well known that any lattice $\zL$ has a \emph{basis}, that is, a linearly independent set of vectors that spans $\zL$, however, it need not be contained in the original set of generators. The \emph{dimension} of the lattice is equal to the dimension of its linear hull.

Since we are not interested in edges that do not belong to any perfect matching, it is convenient to delete from $G$ all edges that are not in any perfect matching, thus making the graph matching covered.
The number of bricks in a tight cut decomposition of a matching covered graph affects the dimension of its perfect matching polytope. In particular, the dimension of the perfect matching polytope is \begin{equation}\label{eq:pm-dimension}
    \dim(PM(G)) = |E'|-|V|+1-b,
\end{equation} where $b$ denotes the number of bricks in a tight cut decomposition of $G$ and $E'$ is the set of edges that belong to a perfect matching \cite{edmonds1982brick}, with $E=E'$ for matching covered $G$. Furthermore, the dimension of the perfect matching lattice and its linear hull is \begin{equation}\label{eq:pm-lat-dimension}
    \dim (\lat(G)) = |E'|-|V|+2-b.
\end{equation}
We will call \emph{near-bricks} all matching covered graphs $G$ with exactly one brick.
The following fact due to Lov\'asz connects the linear hull of perfect matchings with $\zL(G)$: \begin{theorem}[Theorem $6.3$ in \cite{lovaszPM}]\label{thm:lovasz}
    In a matching covered graph $G=(V,E)$, let $\mathcal{M}$ be the set of perfect matchings. Then any $x \in \lin(\mathcal{M})\cap \Z^E$ satisfies $2x\in \zL(G)$. Furthermore, if $G$ has no Petersen bricks then $\lin(\mathcal{M})\cap \Z^E=\zL(G)$. 
\end{theorem}

\subsection{Petersen graph}

In the latter sections, we will reduce lattice basis construction to bricks and braces. In light of \Cref{thm:lovasz}, we describe how to deal with Petersen graph separately.

Suppose that one of the graphs obtained after tight cut decomposition is a Petersen graph $G$ (potentially, with parallel edges). To construct the basis for the Petersen graph, first consider the case where $G$ has no parallel edges. In this case, there are exactly $6$ perfect matchings, as illustrated in \Cref{pic:petersen} and they form a linearly independent set. Via a simple linear algebraic argument, it can be checked that these vectors form a basis for $\zL(G)$.

\begin{figure}[!ht]\label{pic:petersen}
\centering
\begin{tikzpicture}[scale=1.4, every node/.style={circle, draw, fill=black, inner sep=2pt}, spoke/.style={line width=5pt, color=yellow!80!black, opacity=0.5}]

\def\R{2}       
\def\r{1}       
\def\shift{18}

\foreach \i in {1,...,5} {
  \node (o\i) at ({72*(\i-1)+\shift}:\R) {};
}

\foreach \i in {1,...,5} {
  \node (i\i) at ({72*(\i-1)+\shift}:\r) {};
}

\foreach \i [evaluate=\i as \j using {mod(\i,5)+1}] in {1,...,5} {
  \draw (o\i) -- (o\j);
}

\foreach \i/\j in {1/3, 3/5, 5/2, 2/4, 4/1} {
  \draw (i\i) -- (i\j);
}

\foreach \i in {1,...,5} {
  \draw (o\i) -- (i\i);  
}

\foreach \i in {2,...,5} {
  \draw[spoke] (o\i) -- (i\i);
} \draw[spoke] ($(o1)-(0,1pt)$) -- ($(i1)-(0,1pt)$);

\draw[spoke, blue]  ($(o1)+(0,1pt)$) -- ($(i1)+(0,1pt)$);
\draw[spoke, blue] (o2) -- (o3);
\draw[spoke, blue] (o4) -- (o5);
\draw[spoke, blue] (i2) -- (i4);
\draw[spoke, blue] (i3) -- (i5);

\end{tikzpicture}
\caption{The Petersen graph. Its perfect matchings are the central one (yellow) and all possible rotations of the one in blue.}
\end{figure}
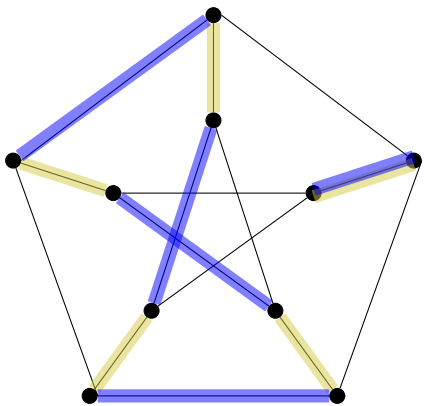

Now, if one adds to an edge $e$ of $G=(V,E)$ a parallel edge $e_2$ to obtain a new graph $G'=(V,E')$, the new basis can be obtained from a lattice basis $\zB$ for $\zL(G)$ as follows.
Pick one perfect matching $M_1 \in \zB$ containing $e$ and let $M^*=M_1\backslash\{e\}\cup \{e_2\}$. Consider the set $\zB'=\zB\cup \{\mathbf{1}_{M^*}\}$. We claim that $\zB'$ is a lattice basis for $\zL(G')$. 
Indeed, it can be readily checked that the set $\zB'$ is linearly independent and $M^*$ is the only element using $e_2$. Furthermore, \eqref{eq:pm-lat-dimension} implies that $\dim(\zL(G'))=\dim (\zL(G'))$ so $\zB'$ is a linear basis. To prove it is integral, consider $y \in \lin(\zB')\cap \Z^{E'}$, so that $y = \sum_{M\in \zB'} \alpha_M \mathbf{1}_M$. Since $y(e_2)$ is integral, it means that $\alpha_{M^*}$ is integral. Hence $y-\alpha_{M^*}\mathbf{1}_{M^*} \in \lin (\zB) \cap \Z^{E}$, implying that it is indeed an integral linear combination of $\zB$, as wanted.

Repeating this replication process until we reach the correct edge multiplicity results in a lattice basis for the target graph and can be done in polynomial time.

\subsection{Tight cut combining}\label{sec:basis-comb}
The following approach describes combining bases for the matching lattices of two matching covered graphs into one using an odd cut. See Section 6.3.1 of \cite{Lucchesi2024} for more details.
\begin{DE}
    Given a matching covered graph $G$, we call an odd cut $C$ \emph{separating} if both $C$-contractions are matching covered.
\end{DE} Equivalently, $C$ is separating if and only if every edge of $G$ belongs to a perfect matching $M$ with $|C\cap M|=1$. All tight cuts are separating, but there could be other examples. 

Let $G$ be a matching covered graph with a non-trivial separating odd cut $C$. Let $G_1$ and $G_2$ be the two $C$-contractions of $G$. Additionally, say that both $G_i$ are matching-covered graphs with $\zB_1=\{\mathbf{x}^1,\mathbf{x}^2,\ldots, \mathbf{x}^{d_1}\}$ and $\zB_2=\{\mathbf{y}^1,\mathbf{y}^2,\ldots, \mathbf{y}^{d_2}\}$ being lattice bases for $\lat(G_1)$ and $\lat(G_2)$ consisting of perfect matchings. 
\begin{DE}[composition]
    For two vectors $\mathbf{v}_1 \in \R^{E_1}$ and $\mathbf{v}_2\in \R^{E_2}$ such that $\mathbf{v}_1(e) = \mathbf{v}_2(e)$ for all $e \in E_1\cap E_2$, define their \emph{composition} $\mathbf{w}:=\mathbf{v}_1\odot \mathbf{v}_2$ as a vector in $\R^{E_1\cup E_2}$ whose restriction on $E_i$ equals $\mathbf{v}_i$ for $i=1,2$. We define \emph{composition} $\zB_1\odot \zB_2$ as follows. For any edge $e \in C$, define $I(e):=\{i: \mathbf{x}^i(e) = 1\}$ and $J(e):=\{j: \mathbf{y}^j(e)=1\}$ to be the indices of elements of $\zB_i$ that contain edge $e$. Write $I(e) = \{i_1,\ldots,i_k\}$ and $J(e) = \{j_1,\ldots,j_\ell\}$ for convenience; notice that $k\geq 1$ and $\ell \geq 1$. Let \begin{align*}
            \mathbf{z}^e_t&:=\mathbf{x}^{i_1} \odot \mathbf{y}^{j_t} \quad &t=1,\ldots,\ell\\
            \mathbf{z}^e_{\ell+t}&:=\mathbf{x}^{i_{1+t}}\odot \mathbf{y}^{j_1}  &t=1,\ldots,k-1.
        \end{align*} 
    We define $\zB_1 \odot \zB_2 = \{\mathbf{z}^e_i:e\in C, 1\leq i\leq |I_e|+|J_e|-1\}$
\end{DE}
Compositions of two lattice bases have useful properties, as shown in e.g. \cite{abdi2025integralbasepm}, \cite{abdi2024str}:
\begin{theorem}\label{thm:basis-comb}
    Let $G$ be a matching covered graph and let $C$ be an odd cut.
    Let $\zB:=\zB_1\odot \zB_2$. Then: \begin{itemize}
        \item[(i)]  $\zB$ consists of linearly independent vectors;
        \item [(ii)] any integral vector in $\lin(\zB)$ can be expressed as an integral linear combination of $\zB$;
        \item [(iii)]$\lin(\zB) = \lin(PM(G)\cap \{x: x(C)=1\})$;
        \item[(iv)] $|\zB|=|\zB_1|+|\zB_2|-|C|$.
    \end{itemize}
\end{theorem}
Since composition of two bases can be done in polynomial time, it suffices to prove Theorem \ref{thm:main} for bricks and braces. Then, to obtain a matching lattice basis for an arbitrary graph $G$, it suffices to consider its tight cut decomposition into (linearly many) bricks and braces $G_1, G_2, \ldots, G_k$ and then compose the resulting bases using Theorem \ref{thm:basis-comb}. The rest of the paper is dedicated to proving the following fact:
\begin{theorem}\label{thm:main-brick}
    There is a polynomial time algorithm that, given a brick or a brace $G$, finds a lattice basis for $\zL(G)$ consisting of incidence vectors of perfect matchings of $G$.
\end{theorem}

\section{A polyhedral approach: Birkhoff von Neumann graphs}\label{sec:bvn}
\subsection{An algorithm for BvN graphs}
In the following sections we discuss a polyhedral approach to basis construction. To demonstrate the main idea, we first look at the graphs $G$ whose perfect matching polytope coincides with its \emph{bipartite relaxation}: \begin{equation}\label{eq:bip}
     P(G):=\left\{ x\in \R^E_{\geq 0}: \begin{array}{cc}
    x(\delta(v))= 1&  \forall v\in V
\end{array}\right\}.
\end{equation} In other words, all points satisfying non-negativity and degree constraints of (\ref{eq:bip}) also satisfy the odd cut constraints of (\ref{eq:pm-general}).  Recall that graphs with this property are called \emph{Birkhoff von Neumann (BvN)}. 

Bipartite graphs are always BvN (see, e.g. 18.1 of \cite{schrijver2003combinatorial}). Non-bipartite graphs are BvN if and only if there are two vertex-disjoint odd cycles whose complement contains a perfect matching due to the fact that every vertex of the bipartite relaxation is half-integral (30.2 in \cite{schrijver2003combinatorial}). Despite this, the problem of testing BvN in polynomial time is still open.

BvN graphs admit the following polyhedral algorithm that continuously reduces the dimension of the polytope.
At each step $t$, we have a set of current edges of the graph $E^t \subseteq E$ which is initially equal to $E$. We consider the polytope $P^t:= \{x \in \R^{E_t}: x(\delta_{G[E_t]}(v)) = 1 \forall v\in V, x\geq \mathbf{0}\}$. Its vertices are precisely the perfect matchings of $G[E_t]$ since $G$ was BvN and $G[E_t]$ is just a restriction of $PM(G)$ to the faces defined by $x_e = 0$ for $e\in E\backslash E_t$. The algorithm will select a facet of the current polytope $P^t$ defined by some $x_e \geq 0$ and go to the facet defined by this inequality. At the same time, augment the current basis $\zB$ by an incidence vector of any perfect matching $M_e$ that contains $e$. The total number of steps this algorithm will take is exactly $\dim(PM(G))$ as the dimension of the current polytope drops by one each time.

To find a facet-defining inequality, it suffices to identify an edge $e$ such that removing the constraint $x_e\geq 0$ from the polyhedral description creates some new feasible points, which inevitably would have $x_e <0$. Therefore, such a point exists if and only if the following linear program has a negative optimal solution: \begin{equation}\tag{LP1$(t,e)$}\label{LP1_facet}
    \min\left\{x_e: x\in \R^{E_t}, \begin{array}{cc}
          x(\delta(v))=1  & \forall v \in V \\
           x_{f}\geq 0 & \forall f \in E_t\backslash \{e\}
       \end{array} \right\} .
\end{equation} 
To find an appropriate matching, one can solve a similar linear program, but maximize the value of $x_e$ instead:
\begin{equation}\tag{LP2$(t,e)$}\label{LP2_match}
    \max\left\{x_e: x\in \R^{E_t}, \begin{array}{cc}
          x(\delta(v))=1  & \forall v \in V \\
           x_{f}\geq 0 & \forall f \in E_t
       \end{array} \right\} .
\end{equation}

Algorithm \ref{algorithm:basis-bvn} summarizes this approach.
\begin{algorithm}
\begin{algorithmic}[1]
	\Require{a BvN graph $G$}
	\Ensure{a set of perfect matchings $\zB$ that forms a basis to $\zL(G)$}
	\State $E^0\gets E$, $P^0\gets PM(G)=P(G)$, $\zB \gets\emptyset$;
	\State for $t=0,1,2, \ldots, \dim(PM(G))-1$:
	\State find $e \in E^t$ such that \ref{LP1_facet} has negative objective value;
    \State let $y(e)$ be a vertex solution to \ref{LP2_match};
    \State define $E^{t+1} = E^t \backslash \{e\}$, $P^{t+1} = P^t \cap \{x: x_e = 0\}$, update $\zB = \zB \cup \{y(e)\}$;
	\State \Return{$\zB$}.
\end{algorithmic}
\caption{BvN basis construction.\label{algorithm:basis-bvn}}	
\end{algorithm}

\begin{theorem}\label{thm:algo-correct}
    The result of Algorithm \ref{algorithm:basis-bvn} applied to a matching covered BvN graph is a set $\zB$ of linearly independent indicator vectors of perfect matchings that forms an integral basis for $\zL(G)$.
\end{theorem}
\begin{proof}
    Because $G$ is BvN it implies that any face of $P(G)$ has integral vertices. Therefore, all $y(e)$ found as a corner solution to \ref{LP2_match} will indeed be incidence vectors of some perfect matchings of $G$ (in fact, of $G[E_t]$). 
    It is clear that the final set $\zB$ will contain exactly $\dim(PM(G))$ perfect matchings with the property that each one uses an edge not used in the previous ones. This means that the elements of $\zB$ are linearly independent and moreover, create an integral basis for $\lin(\zB)$.
\end{proof}

\subsection{Non-BvN detection}\label{subsec:non-bvn-detect}

As mentioned above, the problem of recognizing BvN graphs in polynomial time is still open.
We will avoid directly addressing this problem and instead we will use the following idea which still allows us to use the BvN algorithm for bricks that are non-BvN.

If, given any brick $G$, the above process terminates in $|E|-|V|+1$ steps and all the output vectors are integral, then they correspond to a basis for $\zL(G)$.
Therefore, the only algorithm obstruction in the non-BvN case is \ref{LP2_match} having a fractional solution $x^*$. This means that $x^*$ is a fractional vertex of $P(G)$ itself, which gives a certificate that $G$ is non-BvN.
As $x^*$ has to be half-integral, its support consists of vertex-disjoint edges of weight $1$ and odd cycles of weight $1/2$. Picking any of the odd cycles, we can consider the odd cut $C'$ whose shore is exactly the vertices of this cycle. Inevitably, we have $x^*(C')=0$.
In \Cref{sec:sep-find}, we describe a combinatorial algorithm to find an odd cut $C$ such that the inequality $x(C)\geq 1$ is violated by $x^*$ with more specific properties:
\begin{theorem}\label{thm:find-robust}
    Let $G$ be a brick that is non-BvN. Let $x^*$ be a fractional vertex of $P(G)$. There is a polynomial time algorithm that finds a separating cut $C$ of $G$ such that $x^*(G)<1$ and both $C$-contractions of $G$ are near-bricks whose brick is not the Petersen brick.
\end{theorem}

The following result will prove useful. We will provide the proof of this statement in the appendix.
\begin{theorem}\label{thm:fdi-charact}
    In a near-brick $G$, for an odd cut $C$, the following are equivalent: \begin{itemize}
        \item $C$ is separating and $x(C)\geq 1$ defines a facet of the perfect matching polytope $PM(G)$, and
        \item both $C$-contractions are matching-covered near-bricks.
    \end{itemize} 
\end{theorem}
\section{General graphs}\label{sec:put-together}

As in Section \ref{subsec:non-bvn-detect}, consider a non-BvN brick $G$ for which the BvN algorithm fails. Theorem \ref{thm:find-robust} guarantees the existence of a separating cut $C$ both of whose contractions, call them $G_1$ and $G_2$, have exactly one brick that is not the Petersen brick.
Since neither $G_1$ nor $G_2$ contain the Petersen graph, we can construct lattice bases $\zB_1$ and $\zB_2$ for the corresponding graphs. Since both contractions do not contain Petersen bricks, $\lin(\zB_i)\cap \Z^{E_i} = \zL(G_i)$ and hence by \Cref{thm:basis-comb} the set $\zB:=\zB_1 \odot \zB_2$ will be an integral basis for its linear hull. However, all elements of $\zB$ satisfy $x^T \mathbf{1}_C=1$, so to make it a basis for $\zL(G)$ we need to add at least one perfect matching $M$ with $|M\cap C|>1$.
We use the following fact shown in \cite{camposlucchesi}.
\begin{theorem}\label{thm:char-three}
    In a non-BvN brick that is not Petersen, every separating cut has a perfect matching intersecting it three times.
\end{theorem}
Such a perfect matching $M^*$ can be found in polynomial time by considering all possible intersections $M^*\cap C$ of size three. Finally, we can use $M^*$ to increment the basis. We can now prove that such a set $\zB \cup \{\mathbf{1}_{M^*}\}$ will be a basis for $\zL(G)$.
\begin{theorem}
    Let $G=(V,E)$ be a matching covered graph and let $C$ be a non-tight separating cut such that both $C$-contractions are near-bricks whose brick is not the Petersen brick. Let $M^*$ be a perfect matching with $|M^*\cap C|=3$. Then $\zB \cup \{\mathbf{1}_{M^*}\}$ is an integral basis for $\zL(G)$ consisting of perfect matchings. 
\end{theorem}
\begin{proof}
    By Theorem \ref{thm:fdi-charact}, such $C$ defines a facet of $PM(G)$. Hence, using part $(iii)$ of Theorem \ref{thm:basis-comb}, $|\zB|=\dim(\zL(G))-1$. Thus, adding $M^*$ to $\zB$ makes it a basis for the linear hull of $\zL(G)$, so it suffices to prove that it will be an integral basis.
    Indeed, consider any $y \in \zL(G)$. We also know it is in the linear hull of $\zB\cup\{\mathbf{1}_{M^*}\}$. Therefore, we can write it as follows \begin{equation}\label{eq:linear-comb}
        y = \sum_{M\in \mathcal{M}} \alpha_M \mathbf{1}_M = \sum_{z^i\in \zB} \beta_i z^i+\beta_* \mathbf{1}_{M^*},
    \end{equation} where all $\alpha_M$ are integers and the goal is to prove that $\beta_i$ are all integer.
    The first equality of (\ref{eq:linear-comb}) when multiplied by $\textbf{1}_{E}$ gives \begin{equation*}
        \frac{1}{n/2} \textbf{1}^T_{E} y = \frac{1}{n/2}\sum_{M\in \mathcal{M}} \alpha_M \textbf{1}^T_{E}\mathbf{1}_M=\sum_{M\in \mathcal{M}} \alpha_M \in \Z.
    \end{equation*} Similarly, multiplying it by $\textbf{1}_C$ gives \begin{equation*}
        \textbf{1}^T_{C} y = \sum_{M\in \mathcal{M}} \alpha_M \textbf{1}^T_{C}\mathbf{1}_M\equiv \sum_{M\in \mathcal{M}} \alpha_M \mod 2.
    \end{equation*} Applying the same idea to the last expression of equality (\ref{eq:linear-comb}) we obtain \begin{equation*}
        \frac{1}{n/2} \textbf{1}^T_{E} y =\sum_{\mathbf{z}^i \in \zB} \beta_i + \beta_*, \textbf{1}^T_{C} y = \sum_{\mathbf{z}^i \in \zB} \beta_i + 3\beta_*.
    \end{equation*} Therefore, $\beta_*=\frac{1}{2} (\textbf{1}^T_{C} y-\frac{1}{n/2} \textbf{1}^T_{E} y)$, which is integral by the above.
    Finally, consider $y-\beta_* \mathbf{1}_{M^*} \in \lin(\zB)$. Since we already showed that $\zB$ is an integral basis for its linear hull, all the remaining coefficients $\beta_i$ are also integer.
\end{proof}

\section{Proof of Theorem~\ref{thm:find-robust}}\label{sec:sep-find}

In this section, we prove \Cref{thm:find-robust}. To this end, consider a non-BvN brick $G$ with a half-fractional vertex $x^*$ of $PM(G)$ with $1/2$'s corresponding to some odd cycles of $G$. We can find an odd cut $C$ that has $x^*(C)=0$ and the goal is to find a separating cut $C^*$ for which $x(C^*) \geq 1$ defines a facet of $PM(G)$. A similar process is described in Section 14.1 of \cite{Lucchesi2024}. There are two stages of the process.

\subsection{Finding a separating cut}
If $C$ is separating in $G$, we are done. Otherwise, at least one of the $C$-contractions in $G$ is not matching covered.
Consider this contraction: call it $G_1$ and say
a shore $U$ is contracted into vertex $u$. If it is not matching covered then by Theorem \ref{thm:barrier-find} we can find a barrier $B$ whose vertex set spans an edge set $E'$ (the edges that are not matching covered in $G_1$).  
We may assume $B$ is a maximal barrier. Clearly, $B$ must contain the vertex $u$. Let $K_1, K_2, \ldots, K_t$ be the (odd) components of $G_1-B$ and let $F_i, i=1,\ldots,k$ be the odd cuts defined by these components. Then the following multiset identity holds: $\cup_{v_i\in B\backslash \{u\}}\delta(v_i) \cup C = \cup_{i=1}^t F_i\cup 2 E'$. 
\begin{claim}
    For all $i$, $PM(G)\cap \{x: x(F_i)=1\}\supseteq PM(G)\cap \{x: x(C)=1\}$.
\end{claim}
\begin{cproof}
    For any perfect matching $M$ of $G$ we have \begin{equation*}
    \sum_{i=1}^{t-1}|M\cap \delta(v_i)|+|M\cap C| =\sum_{i=1}^t|M\cap F_i|+2|M\cap E'|. 
    \end{equation*} This reduces to:\begin{equation*}
    |M\cap C|-1 = \sum_{i=1}^t (|M\cap F_i|-1) + 2|M\cap E'|.
    \end{equation*}
    All terms on the right are nonnegative, so $|M\cap C|\geq |M\cap F_i|$ for all $i=1,2,\ldots, k$. Moreover, for the perfect matchings $M$ with $M\cap E' \neq \emptyset$, $|M\cap C|>|M\cap F_i|$.
\end{cproof}
\begin{claim}
    At least one of $F_i$ is non-trivial.
\end{claim}
\begin{cproof}
    Suppose all of $F_i$ were trivial. Then the all components of $\supp(x^*)$ in $G_1$ contain equal number of vertices in $B$ (for edges) as in $V(G_1)-B$ or have more in $B$ (for odd cycles that must inevitably use $E'$). But this is impossible because $u$ is not incident to any edges in $\supp(x^*)$, making $V(G_1)-B$ strictly larger than $B\backslash \{u\}$.
\end{cproof}
\begin{claim}
    For any non-trivial $F_i$, we have $\phi(F_i)<\phi(C)$ where $\phi(S):=\min \{|M\cap S|: e\in M, M\textmd{ perfect matching of } G\}$ defined for all non-trivial odd cuts of $G$.
\end{claim}
\begin{cproof}
    Follows from Claim 1.
\end{cproof}
Thus, after at most $|V|/2$ iterations, we will arrive at a cut whose contractions are matching covered in $G$, i.e., we get a separating cut.
\subsection{Finding a facet defining cut}\label{sub:find-fdi}
Let $C$ be a separating cut.
If both $C$-contractions are near-bricks, then we are done. Suppose one of the $C$-contractions contains more than one brick, say it is $G'$ obtained by shrinking the shore $U$ of $C$ into a vertex $u$. Then $G'$ has a barrier or a $2$-separation cut.



\underline{Case $1$:} $G'$ has a $2$-separation $\{v,w\}$, then clearly one of the vertices has to be $u$, say $w=u$. Let $K_1$ and $K_2$ be two odd components of $G'-u-v$. Let $F_1$ and $F_2$ be two odd cuts induced by $\{u\}\cup F_1$ and $\{u\}\cup F_2$, respectively. Then $\delta(v)\cup C=F_1\cup F_2$. In particular, it implies that for any perfect matching $M$ of $D$, we have \begin{equation*}
    |M\cap \delta(v)|+|M\cap C| = |M\cap F_1|+|M\cap F_2|,
\end{equation*} which reduces to \begin{equation*}
    |M\cap C|-1 = (|M\cap F_1|-1)+(|M\cap F_2|-1).
\end{equation*}
Notice that both terms on the right are nonnegative, hence $|M\cap C|\geq |M\cap F_i|$ for all perfect matchings $M$ of $D$ and $i=1,2$. Furthermore, for any perfect matching $M'$ with $|M'\cap C|= 3$ (whose existence is guaranteed by Theorem \ref{thm:char-three}), we have $|M'\cap C_i|=1$ and $|M'\cap F_{3-i}|=3$ for some $i=1,2$. This means that at least one of $F_1, F_2$ will induce a face of $PM(G)$ that strictly includes the face induced by $C$. 

\underline{Case $2$:} $G'$ has a barrier $B$. We may assume $B$ is a maximal barrier. Clearly, $B$ must contain the vertex $u$. Let $K_1, K_2, \ldots, K_t$ be the (odd) components of $G'-B$ and let $F_i, i=1,\ldots,k$ be the odd cuts defined by these components. Then $\cup_{v_i\in B\backslash \{u\}}\delta(v_i) \cup C = \cup_{i=1}^t F_i$. Similarly to the previous case, for any perfect matching $M$ of $G$ we have the following. \begin{equation*}
    \sum_{i=1}^{t-1}|M\cap \delta(v_i)|+|M\cap C| =\sum_{i=1}^t|M\cap F_i|. 
\end{equation*} This reduces to\begin{equation*}
    |M\cap C|-1 = \sum_{i=1}^t (|M\cap F_i|-1).
\end{equation*}
Notice that all terms on the right are nonnegative. Furthermore, if $M'$ is a perfect matching with $|M'\cap C|=3$ then $|M\cap F_i|=3$ for exactly one of $F_i$. This particular $F_i$ will define a face of $PM(G)$ that is strictly larger than the one defined by $C$.

Finally, the process of finding a stronger odd cut is bounded by the dimension of $PM(G)$. Indeed, adding more vertices to the already existing face of $PM(G)$ means that its dimension increases by at least one. Since the dimension of $PM(G)$ is upper bounded by the number of edges, any sequence of cut updates will take at most $|E|$ steps.

\subsection{Avoiding Petersen bricks after separating cuts}

In the above argument, we require that for the two $C$-contractions of $G$ there is an integral basis for $\lin(G)$ consisting of perfect matchings. In order to guarantee such a property, Theorem \ref{thm:lovasz} requires that neither of the $C$-contractions contains a Petersen brick.
We do this by altering the cut $C$. This procedure was also described in \cite{abdi2025integralbasepm}.

Indeed, let $C$ be a separating cut with shores $X$ and $V\backslash X$ that defines a facet of $PM(G)$. Suppose contracting $X$ into a single vertex $x$ results in a near-brick with a Petersen brick. We may assume that this contraction equals the Petersen graph up to edge multiplicities: by following the steps in \Cref{sub:find-fdi}, $C$ can be adjusted to an equivalent cut whose contraction on a specific side does not have tight cuts. Label the vertices of the $X$-contraction as $x,y,z,v,w$ in the outer $5$-cycle and $x',y',z',v',w'$ as the inner vertices. Then, consider the cut $\overline{C}$ with a shore $\overline{X}:=V\backslash X\cup \{y,y',w,w'\}$. 
\begin{claim}
    $\overline{C}$ is a separating cut whose contractions are near-bricks.
\end{claim}
\begin{cproof}
    It is straightforward to check that the graph obtained from Petersen by contracting the vertices corresponding to $\{x,y,y',w,w'\}$ is a matching covered near-brick (it is just a $5$-wheel). Thus, it suffices to check the same for the graph obtained from $G$ by shrinking $V\backslash\overline{X}$.

    To prove it, we construct a large set of linearly independent perfect matchings of $G_2'$. Consider the basis $B_2$ for $G_2$ and extend each of its elements to a perfect matching of $G_2'$. Furthermore, add a perfect matching of $G_2'$ that intersects $C$ three times and intersects $C'$ once. In total, we obtain $\dim(G_2)+1$ linearly independent perfect matchings, which cover each edge of $G_2'$, so $C'$ is indeed separating. Moreover, $\dim(G_2')\geq \dim(G_2)+1$. On the other hand, $G_2$ is a near-brick, so $\dim(G_2) = |E_2|-|V_2|+1$ and so $|E_2'|-|V_2'|+1-b(G_2') = |E_2|+5-|V_2|-4+1-b(G_2') \geq |E_2|-|V_2|+2$, meaning $b(G_2')\leq 1$. Clearly, it cannot be less than $1$ as $G_2'$ is non-bipartite, so the only possibility is $b(G_2')=1$, as wanted.
\end{cproof}

\section{Acknowledgments}
I am grateful to Ahmad Abdi, G\'erard Cornuejols, and Siyue Liu for their invaluable discussions and encouragement throughout this project.

\bibliographystyle{splncs04}
\bibliography{bibliography}

@article{lovaszPM,
  title={Matching structure and the matching lattice},
  author={Lov{\'a}sz, L{\'a}szl{\'o}},
  journal={Journal of Combinatorial Theory, Series B},
  volume={43},
  number={2},
  pages={187--222},
  year={1987},
  publisher={Elsevier}
}

@article{Seymour,
author = {Seymour, P. D.},
title = {On Multi-Colourings of Cubic Graphs, and Conjectures of Fulkerson and Tutte},
journal = {Proceedings of the London Mathematical Society},
volume = {s3-38},
number = {3},
pages = {423-460},
doi = {https://doi.org/10.1112/plms/s3-38.3.423},
url = {https://londmathsoc.onlinelibrary.wiley.com/doi/abs/10.1112/plms/s3-38.3.423},
eprint = {https://londmathsoc.onlinelibrary.wiley.com/doi/pdf/10.1112/plms/s3-38.3.423},
year = {1979}
}

@article{dimension,
author = {Carvalho, Marcelo and Lucchesi, Claudio and Murty, Uppaluri},
year = {2002},
month = {05},
pages = {59-93},
title = {Optimal Ear Decompositions of Matching Covered Graphs and Bases for the Matching Lattice},
volume = {85},
journal = {J. Comb. Theory, Ser. B},
doi = {10.1006/jctb.2001.2090}
}

@book{lovaszplummer,
  title={Matching Theory},
  author={Lov{\'a}sz, L. and Plummer, M.D.},
  isbn={9780821847596},
  lccn={2009007644},
  series={AMS Chelsea Publishing Series},
  url={https://books.google.com/books?id=OaoJBAAAQBAJ},
  year={2009},
  publisher={AMS Chelsea Pub.}
}

@article{conj-lovasz,
author = {de Carvalho, Marcelo H. and Lucchesi, Cl\'{a}udio L. and Murty, U.S.R.},
title = {On a Conjecture of {Lov\'{a}sz} Concerning Bricks},
year = {2002},
issue_date = {May 2002},
publisher = {Academic Press, Inc.},
address = {USA},
volume = {85},
number = {1},
issn = {0095-8956},
url = {https://doi.org/10.1006/jctb.2001.2091},
journal = {J. Comb. Theory Ser. B},
month = may,
pages = {94–136},
numpages = {43}
}

@book{edmonds1982brick,
  title={Brick Decompositions and the Matching Rank of Graphs},
  author={Edmonds, J. and Lov{\'a}sz, L. and Pulleyblank, W.R.},
  series={Research report},
  url={https://books.google.com/books?id=ZG1KNAEACAAJ},
  year={1982},
  publisher={Department of Combinatorics \& Optimization, University of Waterloo}
}

@article{edmonds1965maximum,
  author = {Edmonds, Jack},
  journal = {Journal of Research of the National Bureau of Standards B},
  pages = {125--130},
  title = {Maximum matching and a polyhedron with 0, 1-vertices},
  volume = 69,
  year = 1965
}

@article {LovaszOntheStr,
      author = "L. Lovász",
      title = "On the structure of factorizable graphs",
      journal = "Acta Mathematica Hungarica",
      year = "1972",
      publisher = "Akadémiai Kiadó, co-published with Springer Science+Business Media B.V., Formerly Kluwer Academic Publishers B.V.",
      address = "Budapest, Hungary",
      volume = "23",
      number = "1-2",
      doi = "10.1007/bf01889914",
      pages=      "179 - 195",
      url = "https://akjournals.com/view/journals/10473/23/1-2/article-p179.xml"
}

@misc{abdi2025integralbasepm,
      title={Integral bases, perfect matchings, and the {Petersen} graph}, 
      author={Ahmad Abdi and Olha Silina},
      year={2025},
      eprint={2508.15602},
      archivePrefix={arXiv},
      primaryClass={math.CO},
      url={https://arxiv.org/abs/2508.15602}, 
}

@book{Lucchesi2024,
author="Lucchesi, Cl{\'a}udio L.
and Murty, U. S. R.",
title="Perfect Matchings ",
bookTitle="Perfect Matchings: A Theory of Matching Covered Graphs",
year="2024",
publisher="Springer Nature Switzerland",
address="Cham",
pages="3--22",
isbn="978-3-031-47504-7",
doi="10.1007/978-3-031-47504-7_1"
}

@book{schrijver2003combinatorial,
  title={Combinatorial optimization: polyhedra and efficiency},
  author={Schrijver, Alexander},
  volume={24},
  year={2003},
  publisher={Springer}
}

@techreport{camposlucchesi,
  number = {IC-00-22},
  author = {Christiane N. Campos and Claudio L. Lucchesi},
  title = {On the Relation between the {Petersen} Graph and the Characteristic of Separating Cuts of Matching Covered Graphs},
  month = {December},
  year = {2000}, 
  institution = {Institute of Computing, University of Campinas}
}

@article{X,
 author = {Abdi, Ahmad and Cornu\'{e}jols, G\'{e}rard and Liu, Siyue and Silina, Olha},
 title = {Strongly Connected Orientations and&nbsp;Integer Lattices},
 year = {2025},
 publisher = {Springer-Verlag},
 url = {https://doi.org/10.1007/978-3-031-93112-3_1},
 doi = {10.1007/978-3-031-93112-3_1},
 journal = {IPCO 2025, Proceedings},
 pages = {1–14},
 numpages = {14},
 location = {Baltimore, MD, USA}
}

@inproceedings{abdi2024str,
  title={Strongly Connected Orientations and Integer Lattices},
  author={Ahmad Abdi and G\'erard Cornu\'ejols and Siyue Liu and Olha Silina},
  booktitle={Integer Programming and Combinatorial Optimization},
  year={2025},
  url={https://doi.org/10.1007/978-3-031-93112-3_1}
}

\appendix
\section{Connection to ear decompositions}

\subsection{Classical approach}\label{sec:classical}

We begin by describing the approach of Carvalho, Lucchesi, and Murty in \cite{dimension}. 
One way of constructing a basis for bricks is through \emph{ear decompositions}. Given a matching covered graph $G$, a \emph{single ear} is a path $P=v_1 v_2 \ldots v_k$ with $k$ even where $v_i$ has degree two for every $i = 2, 3, \ldots, k-1$. A \emph{double ear} consists of two vertex-disjoint single ears. Finally, an \emph{ear} $R$ is either a single or a double ear; $R$ is \emph{removable} if $G-E(R)$ is matching covered. It is known (\cite{lovaszplummer}) that each matching covered graph has an \emph{ear decomposition}, which is a sequence $K_2 = G_0\subset G_2 \subset \ldots \subset G_d = G$ where $G_i$ is obtained from $G_{i+1}$ by removing an ear. It is important to distinguish between single and double ears, as those affect the dimension of the perfect matching lattice differently. In bricks, all ears consist of a single edge as otherwise its endpoints would create a $2$-separation cut.

The basis construction is based on finding an ear decomposition where each ear $R_{i+1}:=G_{i+1}-G_i$ is removable in $G_{i+1}$. Once this is done, we may inductively construct a large set of linearly independent perfect matchings of $G$ as follows. At step $i$, let $\zB$ be the current set of matchings (for $G_i$), consider $R_{i+1}$. We may assume that $R_{i+1}$ is either an edge or a pair of edges since any ear $v_1,v_2, \ldots, v_k$ for even $k$ induces a tight cut with a shore $\{v_1,v_2, \ldots, v_{k-1}\}$. For an ear $R_{i+1}$, consider the set $\zB':=\{M\cup X: M \in \zB\} \cup \{M_Y\}$ where $M_Y$ is any perfect matching using an edge of $Y$ (notice that this forces it to contain all of $Y$). Each element of this set is a perfect matching of $G_{i+1}$ and the characteristic vectors of these matchings form an integral basis for its linear hull. The last statement is due to $M_Y$ being the only matching containing an edge of $Y$ and by the inductive hypothesis on $\zB$.

If the number of double ears is appropriate, the set $\zB$ obtained in the end will have the same dimension as $\zL(G)$. To ensure this, it can be shown that every non-Petersen brick has an ear decomposition with exactly one ear. We use the following theorem due to Carvalho, Lucchesi, and Murty:
\begin{theorem}[Theorem $3.3$ in \cite{conj-lovasz}]\label{thm:brick-invariant-edge}
    Every brick $G$ that is not Petersen or $K_4$ or $\overline{C_6}$ has an edge $e$ such that $G-e$ is a matching covered near-brick that does not contain a Petersen brick. 
\end{theorem}

This theorem implies the existence of an ear decomposition with exactly one double ear as follows: let $G$ be a non-Petersen brick in question. If it equals $K_4$, the basis is obtained by considering all three of its perfect matchings. If $G$ is not Petersen or $K_4$, consider the edge $e$ found in \cref{thm:brick-invariant-edge}. Then $\dim(PM(G-e)) = |E\backslash \{e\}|-|V|+1-b(G-e) = \dim(PM(G))-1$ since $b(G-e)=1$. Moreover, since $G-e$ does not contain the Petersen graph as a brick, $G-e$ still has a set of perfect matchings whose incidence vectors form an integral basis for the linear hull of perfect matchings of $G-e$. Therefore, augmenting it with any perfect matching using $e$ will create a linearly independent set that forms an integral basis for $G$. To proceed, it suffices to contract all tight cuts if necessary and consider the only brick in the list. 

\subsection{Bipartite BvN graphs}

First, we establish a connection between the polyhedral description of the polytope and the combinatorial properties of its edges.
\begin{theorem}\label{thm:facets-bip}
    Let $G$ be a bipartite matching covered graph that is not a $4$-cycle $C_4$. The following are equivalent: \begin{itemize}
        \item[(i):] the inequality $x_e\geq 0$ defines a facet of the perfect matching polytope $PM(G)$ for an edge $e$,
        \item[(ii):] $G-e$ is a matching covered bipartite graph. 
    \end{itemize}
\end{theorem}
\begin{proof}
    To prove $(i) \Rightarrow (ii)$, first notice that $\dim(PM(G))=|E'|-|V|+ \sigma(G[E'])+1$ where $E' \subseteq E\backslash \{e\}$ is the set of edges that belong to some perfect matching of $G-e$. As this face is a facet of $PM(G)$ its dimension is exactly one smaller, which means \begin{equation*}
        |E'|-|V|+ \sigma(G[E'])+1 = |E|-|V|+2 -1,
    \end{equation*} implying $|E'|+\sigma(G[E'])=|E|$. By $2$-connectivity of $G$, the graph $G-e$ is connected and so the maximal number of components of $G[E']$ is $|E|-|E'|$. If $|E|-|E'|=1$ then every edge except $e$ is in $E'$, meaning $G-e$ is matching covered. Otherwise, let $f \in E\backslash E'$ and $F\neq e$. We must have that $\sigma(G-e-f)=2$, so let $V_1$ and $V_2$ be the two vertex sets separated by the cut $\{e,f\}$. We know that either $V_i$ have both even cardinality or both odd. In the first case, since $G$ is not $C_4$, at least one of $V_i$ has size at least $4$, say $V_1$.  the cut defined by $V_1 \backslash \{u\}$ where $u$ is the endpoint of $e$ inside $V_1$ will make a tight cut. In the second case, either one of the sets will have size at least $5$ or both will have size at least $3$. If $V_1$ has at least $5$ vertices, the set $V_1 \backslash \{u,w\}$ will define a tight cut where $u$ and $w$ are the endpoints of $e,f$ within $V_1$. If both sets have size at least $3$, the cut $\{e,f\}$ will be non-trivial and tight in $G$, which is not possible.
    
    Conversely, if $G-e$ is matching covered, the dimension of $PM(G-e)$ is exactly $|E|-1-|V|+2$, which means the dimension of the face of $PM(G)$ defined by $x_e=0$ is exactly one smaller than the dimension of $PM(G)$. Thus $(ii) \Rightarrow (i)$ holds.
\end{proof}

\begin{CO}
    The edge $e$ as in \Cref{thm:facets-bip} satisfies the definition of a removable edge from \Cref{sec:classical}, so the Algorithm \ref{algorithm:basis-bvn} described here will also give an ear-decomposition of a bipartite graph.
\end{CO}

\subsection{Nonbipartite BvN graphs}

We begin with a useful characterization. \begin{theorem}\label{thm:facet-nonbip}
    Let $G$ be a BvN brick. If an inequality $x_e\geq 0$ is a facet defining inequality for $PM(G)$ then one of the following holds: \begin{itemize}
        \item[(a)] $G-e$ is perfect matching covered and is a near-brick;
        \item[(b)] there is a unique edge $e'\neq e$ such that $G-e-e'$ is bipartite and $x_e = x_{e'}$ for all $x \in PM(G)$.
    \end{itemize}
    Furthermore, if $G$ is a BvN brick with an edge $e$ such that (a) or (b) holds, then $x_e\geq 0$ is a facet defining inequality for PM(G).
\end{theorem}
\begin{proof}
    We proceed by computing the affine dimension of $PM(G)$ and its facet. Notice that the equalities $\delta(v)=1$ are either independent or have rank $|V|-1$ if $G$ is bipartite. Thus, the dimension of (\ref{eq:pm-bipartite}) is exactly $|E|-|V|+\sigma_b(G)=|E|-|V|$, where $\sigma_b(H)$ denotes the number of bipartite connected components of $H$. 
    Now, consider a face defined by setting $x_e=0$. All vertices of this face are exactly the perfect matchings that do not use $e$ or, in other words, the perfect matchings of $G-e$. If $G-e$ is not perfect matching covered, then there is at least one edge $e'$ which is no longer in any perfect matching. Let $E'$ be the set of edges that belong to some perfect matching of $G-e$. Then, the vertices of the face defined by $x_e=0$ are exactly the perfect matchings of $G[E']$, so $PM(G[E'])$ has dimension $|E'|-|V|+\sigma_b(G[E'])$. Since this is one smaller than the original dimension, we get the following \begin{equation*}
        \sigma_b(G[E'])=|E\backslash E'|-1.
    \end{equation*} Let us denote this quantity by $k$. This means that by removing $k+1$ edges from $G$, the graph has at least $k$ bipartite components left. We will prove that unless $k\geq 2$ there is a $2$-separation in $G$, contradicting the fact that $G$ was a brick.

    Indeed, let $C_1, C_2, \ldots, C_k$ be the bipartite components. First, suppose $|\delta(C_i)|\leq 2$ for some component. It cannot be less than $2$ since it would mean that in $G$, there is a cut-edge, so it does not belong to any perfect matching. Then, let $u\in C_i$ and $v\notin C_i$ be the endpoints of two different edges leaving $C_i$. It is easy to see that $G-u-v$ has at least two odd components, so $\{u,v\}$ forms a $2$-separation. Now, we have $\delta(C_i)\geq 3$. Counting this over all components, we have \begin{equation}\label{eq:components-count}
        3k\leq \sum_i |\delta(C_i)|=2|E\backslash E'|-N \leq 2(k+1),
    \end{equation} where $N=0$ if $C_i$ are the only components and $N=\sum_j \delta(D_j)$ if $G[E']$ has more non-bipartite components $D_j$ (notice that there are at most two of those). From (\ref{eq:components-count}) we deduce $k\leq 2$.

    \underline{$k=2:$} in this case, (\ref{eq:components-count}) holds with equality, so there are no non-bipartite components. Let $A_i,B_i$ be the bipartition of $C_i$ for $i=1,2$. There are three edges $e,e_1,e_2$ in $E\backslash E'$ and $x_e=0$ implies $x_{e_1}=x_{e_2}=0$ for all $x \in PM(G)$.
    Without loss of generality, let $e$ connect $A_1$ to $B_2$. Then both $e_1$ and $e_2$ must connect $B_1$ to $A_2$, as otherwise there cannot be a perfect matching using $e$ and one of $e_1, e_2$. However, this makes the original graph bipartite, which is not possible.
    
    \underline{$k=1$:} this means that $E'=E\backslash\{e, e_1\}$ and $G-e-e_1$ is a perfect matching covered bipartite graph. Moreover, $e$ and $e_1$ are inside the opposite sides of the bipartition.

    \underline{$k=0$:} this means $G-e$ is perfect matching covered.

    To prove the converse, it suffices to redo the dimension calculations from the above part.
\end{proof}

\begin{CO}
    An edge $e$ or a pair $\{e,e'\}$ found in \Cref{thm:facet-nonbip} correspond to a removable single ear or a removable double ear from \Cref{sec:classical}. Thus, the Algorithm \ref{algorithm:basis-bvn} described above also gives an ear decomposition of the graph.
\end{CO}

\section{Proof of Theorem \ref{thm:fdi-charact}}
\begin{proof}
    Denote by $G_1$ and $G_2$ two $C$-contractions of $G$.
    The vertices on the face $F_C:=P(G)\cap \{x: x(C)=1\}$ are exactly the perfect matchings that intersect $C$ once. Therefore, there is a bijection between such perfect matchings and pairs $(M_1,M_2)$ of perfect matchings of $G_1,G_2$ with $M_1\cap M_2 = \{e\}$ for some $e \in C$.
    Therefore, we have the following dimension formula as a consequence of Theorem \ref{thm:basis-comb}: \begin{equation}
        \dim(F_C)=\dim(P(G_1))+\dim(P(G_2))+1-|C|.
    \end{equation}
    By (\ref{eq:pm-dimension}):\begin{equation*}
        \dim(F_C) = |E_1'|-|V_1|+1-b(G_1)+|E_2'|-|V_2|+1-b(G_2)+1-|C|,
    \end{equation*} where $E_i'$ are matching-covered edges of $G_i$. For convenience, let $|E_i|-|E_i'|=r(E_i)\geq 0$ the edges which do not belong to a perfect matching in $G_i$. 
    
    We have $|E_1|+|E_2|=|E|+|C|$ and $|V_1|+|V_2|=|V|+2$, so \begin{align}\label{eq:dimension-add}
        \dim(F_C)&=|E|-|V|+1-b(G_1)-b(G_2)-r(G_1)-r(G_2) \\&= \dim(P(G))-1 + (2-b(G_1)-b(G_2)-r(G_1)-r(G_2)).
    \end{align}

    To prove $\Rightarrow$, assume that $C$ defines a facet of $PM(G)$, so $\dim(F_C)=\dim(P(G))-1$ and therefore \begin{equation*}
        2 = b(G_1)+b(G_2)+r(G_1)+r(G_2).
    \end{equation*} Since $C$ is not equivalent to $x_e=0$ for any $e\in E$, there is at least one vertex in $F_C$ with $x_e=1$, so each $e$ belongs to some perfect matching $M$ of $G$ with $|M\cap C|=1$. Its projections to $G_1$ and $G_2$ will correspond to perfect matchings and thus $r(G_i)=0$. Therefore, $2 = b(G_1)+b(G_2)$. It suffices to prove $b(G_i)\geq 1$. Indeed, if one of the terms was zero, say $b(G_1)=0$, it would mean that one of the $C$-contractions was bipartite and thus $x(C)$ can be written as a linear combination of $x(\delta(v))$. But this would imply that $C$ is a tight cut and it cannot define a facet of $G$.

    To prove $\Leftarrow$, if both $G_i$ are matching-covered near-bricks, then (\ref{eq:dimension-add}) implies $\dim(F_C)=\dim(P(G))-1$. Moreover, every $e\in E$ belongs to some perfect matching of $G$ with $|M\cap C|=1$ by combining appropriate perfect matchings of $G_i$, so $F_C$ is not contained in $\{x:x_e=0\}$.
\end{proof}

\end{document}